\documentclass{amsart}
\usepackage{graphicx,tikz}
	\usetikzlibrary{backgrounds,fit,matrix}	
\usepackage{amssymb}
\usepackage{amsrefs}
\usepackage[mathscr]{euscript}
\usepackage{paralist}

\theoremstyle{plain}
\newtheorem{thm}{Theorem}
\newtheorem*{thm*}{Theorem}
\newtheorem*{prop*}{Proposition}
\newtheorem{lem}{Lemma}
\newtheorem{cor}{Corollary}
\newtheorem*{cor*}{Corollary}
\newtheorem{conj}{Conjecture}
\theoremstyle{remark}
\newtheorem{ex}{Example}
\newtheorem{rmk}{Remark}

\newcommand{\pr}{\mathfrak P} % prinicipal minors ideal
\DeclareMathOperator{\Var}{\mathscr V} % variety
\DeclareMathOperator{\I}{I} % Hochster-Eagon ideal of minors
\DeclareMathOperator{\hgt}{ht} % height of an ideal
\newcommand{\Prj}{\mathbb P} % projective space
\DeclareMathOperator{\cof}{cof}
\DeclareMathOperator{\Adj}{adj}
\DeclareMathOperator{\Spec}{Spec}
\DeclareMathOperator{\GL}{GL}
\DeclareMathOperator{\rank}{rank}
\newcommand{\transpose}{^{\rm T}} % matrix transpose
\newcommand{\vect}[1]{\mathbf{#1}}
\newcommand{\Z}{\mathbb Z}
\newcommand{\N}{\mathbb N}
\DeclareMathOperator{\Grass}{Grass}

\begin{document}
% % %
% % %
\title{Ideals Generated by Principal Minors}
% % %
\author{Ashley K. Wheeler}
\address{Department of Mathematical Sciences \\
	University of Arkansas \\
	Fayetteville, AR 72701
	}
\email{ashleykw@uark.edu}
\date{\today}
% % %
\thanks{This work partially partially supported by NSF grant 0943832.}
% % %
\begin{abstract}
A minor is \emph{principal} means it is defined by the same row and column indices.  We study ideals generated by principal minors of size $t\leq n$ of a generic $n\times n$ matrix $X$, in the polynomial ring generated over an algebraically closed field by the entries of $X$.  When $t=2$ the resulting quotient ring is a normal complete intersection domain.  We show for any $t$, upon inverting $\det X$ the ideals given respectively by the size $t$ and the size $n-t$ principal minors become isomorphic.  From that we show the algebraic set given by the size $n-1$ principal minors has a codimension $4$ component defined by the determinantal ideal, plus a codimension $n$ component.  When $n=4$ the two components are linked, and in fact, geometrically linked. 
\end{abstract}
\maketitle

% % % % % % % % % % % % % % % % % % % %
%%% Body
% % % % % % % % % % % % % % % % % % % %

\section{Introduction}

Let $X=(x_{ij})$ denote a generic matrix, $K[X]$ the polynomial ring over an arbitrary algebraically closed field $K$, generated as a $K$-algebra by the entries $x_{ij}$.  For the past several decades algebraists have studied ideals defined using generic matrices -- there are the very well-studied determinantal ideals (see ~\cites{room, eagon_thesis, eagon+northcott, sharpe, hochster+eagon, svanes, hochster+huneke/94_2, harris+tu84}), due to their connection to invariant theory (as in \cite{deconcini+procesi76}); Pfaffian ideals (see ~\cites{kleppe, jozefiak+pragacz79, kleppe+laksov80, pragacz81, denegri+gorla11}), whose study is often inspired by the result from \cite{buchsbaum+eisenbud}; and various ideals related to the commuting variety (see ~\cites{gerstenhaber61, richardson79, baranovsky01, knutson, mueller_thesis, young_thesis}), to name only a few.  

Our focus is on the ideals generated by the principal minors (i.e., those whose defining row and column indices are the same) of an $n\times n$ generic matrix.  We use $\pr_t=\pr_t(X)\subseteq K[X]$ to denote the ideal generated by the size $t$ principal minors of $X$.  In developing their generalized version of the Principal Minor Theorem, Kodiyalam, Lam, and Swan (\cite{kodiyalam+lam+swan}) draw some relationships between the principal minor ideals and the Pfaffian ideals.  One contrast that stands out is that the Pfaffian ideals, like the determinantal ideals, satisfy a chain condition according to rank, whereas the principal minor ideals do not.  

Our motivation is in finding results about principal minor ideals comparable to what is known about the Pfaffian and determinantal ideals.  As the observations in \cite{kodiyalam+lam+swan}, and in \cite{ghorpade+krattenthaler04}, suggest, such results prove challenging.  For example, it turns out the ideals $\pr_t$ are in general not Cohen-Macualay (follows from Theorem \ref{thm:n-1}).  By Hochster and Roberts' famous result that rings of invariants of reductive groups acting on regular rings, a class in which quotients by determinantal and Pfaffian ideals belong, are Cohen-Macaulay, it follows that principal minor ideals do not arise as defining ideals for rings of invariants.

Nonetheless, principal minors have their own interesting contexts.  In the late nineteenth and early twentieth centuries algebraists began studying subsets of all the principal minors of a square matrix; see ~\cite{macmahon, stouffer24, stouffer28}, in which an established result is that $n^2-n+1$ of the $2^n-1$ principal minors of a generic $n\times n$ matrix are independent.  In \cite{griffin+tsatsomeros05}, Griffin and Tsatsomeros study the Principal Minor Assignment Problem (PMAP), the problem of finding an $n\times n$ matrix with prescribed principal minors.  Among other applications, PMAP ``is a natural algebraic problem with many potential applications akin to inverse eigenvalue and pole assignment problems that arise in engineering and other mathematical sciences."  Holtz and Sturmfels show the relations among the principal minors of a symmetric matrix satisfy the \emph{hyperdeterminantal relations} (see \cite{holtz+sturmfels07}), and comment on the connections to probability theory.  The set-theoretic equations for the variety of principal minors of symmetric $n\times n$ matrices are computed in \cite{oeding11_2}; the same year Oeding shows in \cite{oeding11_1} that restricting to matrices of rank at most one gives the tangential variety of the Segre variety.  

Again, in this work we focus on principal minors of a fixed size.  Ideals generated by subsets of 2-minors alone have applications in integer programming because they are binomial (see \cite{diaconis+eisenbud+sturmfels98}); it is known, as a direct consequence of the results in \cite{ene+qureshi}, that $\pr_2$ is a prime complete intersection, its corresponding quotient ring is normal, and its divisor class group is free of rank $2^n-n-\binom{n}{2}$.

% % %
\subsection{Results}

We collect and summarize the main results. 

\begin{thm*}[Theorems \ref{thm:p2X} and \ref{thm:2normal}, Corollary \ref{cor:f-reg}; see Section \ref{sec:p2X}] 
For all $n$, $\pr_2$ is prime, normal, a complete intersection, and toric.  Hence, $\pr_2$ is strongly $F$-regular and Gorenstein of codimension $\binom{n}{2}$. 
\end{thm*}

The strategies in proving Theorems \ref{thm:p2X} and \ref{thm:2normal} heavily exploit the fact that $\pr_2$ is toric (see Chapter 4 of \cite{sturmfels/96} for information about toric ideals), that is, the quotient $K[X]/\pr_2$ is isomorphic to a ring generated by monomials.  Our proof of these results is independent of those given in \cite{ene+qureshi}.  We use Cohen-Macaulayness of a complete intersection to prove primality of $\pr_2$.  

The $t>2$ cases require a different approach.  We work based on the observation that the irreducible components of $\Var(\pr_t)$ for fixed $t$ may be partitioned according to the rank of a generic element.  Let $\mathscr Y_{n,r,t}$ denote the locally closed set of rank $r$ matrices in $\Var(\pr_t)$.  As $r$ varies, the components of the sets $\mathscr Y_{n,r,t}$ cover $\Var(\pr_t)$, hence, so do their closures.  Our approach to studying the schemes $\Var(\pr_t)$ is to find those closures, and then omit the ones not maximal in the family, in order to get the components of $\Var(\pr_t)$.

\begin{thm*}[\ref{thm:invBij}; see Section \ref{sec:rEqn}] 
In the localized ring $K[X][\frac{1}{\det X}]$, the $K$-algebra automorphism $X\mapsto X^{-1}$ induces an isomorphism of schemes $\mathscr Y_{n,n,t}\cong\mathscr Y_{n,n,n-t}$. 
\end{thm*}

\begin{thm*}[\ref{thm:n-1}; see Section \ref{sec:n-1}]  
For $n\geq 4$, $\Var(\pr_{n-1})$ has two components.  One is defined by the determinantal ideal $\I_{n-1}$.  The other is the Zariski closure of the locally closed set $\mathscr Y_{n,n,n-1}$, and has codimension $n$.
\end{thm*}

\begin{cor*}[\ref{cor:n-1}; see Section \ref{sec:n-1}] 
For $n\neq 3$, $\hgt(\pr_t)\leq\binom{n+1}{2}-\binom{t+2}{2}+4$. 
\end{cor*}

We conjecture $\pr_{n-1}$ is reduced and we prove it for $n=4$.  It follows that $\I_3=\I_3(X_{4\times 4})$ and $\mathfrak Q_3=\mathfrak Q_3(X_{4\times 4})$, the defining ideal for the Zariski closed set $\overline{\mathscr Y}_{4,4,3}$, are algebraically linked.  In Section \ref{sec:ci} we discuss these consequences.  

\begin{thm*}[Theorem \ref{thm:red}, Corollaries \ref{cor:QCM}, \ref{cor:fQ}, \ref{cor:definingQ}; see Section \ref{sec:ci}] 
For $n=4$, $\pr_3$ is reduced.  Consequently, $\I_3$ and $\mathfrak Q_3$ are algebraically linked and hence $\mathfrak Q_3$ is Cohen-Macaulay with $5$ generators.  Furthermore, $\mathfrak Q_3=\pr_3:_{K[X]}\Delta$, where $\Delta=\det X$. 
\end{thm*}

% % %
\subsection{Acknowledgements}

The author would like to thank Viviana Ene and Luke Oeding for pointing out their results on principal minors, and Lance Miller, for the observation that in general, since a quotient by a principal minor ideal is not Cohen-Macaulay, it cannot be a ring of invariants.

% % % % %
\section{Principal 2-Minors Case}\label{sec:p2X}

For all $n$, we immediately see $K[X]/\pr_1$ is isomorphic to a polynomial ring in $n^2-n$ variables over $K$, since the generators for $\pr_1$ are just the diagonal entries of $X$.  For the $n=t=2$ case, $X=X_{2\times 2}$, we recognize $K[X]/\pr_2$ as the homogeneous coordinate ring of the image of $\Prj^1\times\Prj^1\hookrightarrow\Prj^3$ under the Segre embedding.

The next case we consider is when $t=2$ and $n$ is fixed.  We first show $R=K[X]/\pr_2$ is a complete intersection and $\pr_2$ is toric, hence prime.  We then show $R$ is normal.  Hochster showed in \cite{hochster/72} (1972) that quotients of toric ideals are direct summands of polynomial rings, a fact that, together with normality, implies when $K$ is characteristic $p>0$, $K[X]/\pr_2$ is $F$-regular (\cite{hochster+huneke/90}).  Furthermore, since $R$ is Gorenstein all notions for $R$ of $F$-regularity, strong $F$-regularity, and weak $F$-regularity are equivalent (\cite{hochster+huneke/94_2}).   

\begin{thm}\label{thm:p2X} 
For all $n$, $K[X]/\pr_2$ is a complete intersection domain. 
\end{thm}

\begin{proof} 
The $n=2$ case gives the homogeneous coordinate ring for $\Prj^1\times\Prj^1$, which is a complete intersection domain.  We proceed by induction on $n$: let $X'=(u_{ij})_{1\leq i,j\leq n-1}$ denote a size $n-1$ matrix of indeterminates and suppose $\pr_2(X')$ satisfies the theorem.  Append to the bottom of $X'$ the row $(x_1 \cdots x_{n-1})$, then to the far right the column $(y_1,\dots,y_{n-1},z)$.  Let $X$ denote the resulting size $n$ matrix.  The ideal generated by the principal 2-minors is
\[
\pr_2(X)=\pr_2(X')K[X]+(zu_{ii}-x_iy_i \mid i=1,\dots,n-1).
\]
Put $A=K[X']/\pr_2(X')$, and
\[
R=\frac{A[x_1,\dots,x_{n-1},z,y_1,\dots,y_{n-1}]}{(zu_{ii}-x_iy_i)}\cong\frac{K[X]}{\pr_2(X)}.
\]

First, we show $R$ is a complete intersection.  By the induction hypothesis, $A$ is a complete intersection domain, hence so is the polynomial ring 
\[
\tilde A=A[x_1,\dots,x_{n-1},z,y_1,\dots,y_{n-1}].
\]  
In particular, $\tilde A$ is Cohen-Macaulay.  It follows that if $zu_{ii}-x_iy_i$, for $i=1,\dots,n-1$, form a regular sequence on $\tilde A$, then $R$ is also a complete intersection.  We shall show $x_iy_i$ form a regular sequence in $\tilde A/(z)$.  To see why, note all polynomials we consider are homogeneous, so $zu_{ii}-x_iy_i$ form a regular sequence if and only if 
\[
zu_{11}-x_1y_1,\dots,zu_{n-1,n-1}-x_{n-1}y_{n-1},z
\] 
form a regular sequence, if and only if 
\[
z,zu_{11}-x_1y_1,\dots,zu_{n-1,n-1}-x_{n-1}y_{n-1}
\] 
form a regular sequence.  The strategy works because clearly $z$ is not a zerodivisor in the domain $\tilde A$.  Working now in $\tilde A/(z)$, there are $2^{n-1}$ minimal primes for the ideal $I=(x_iy_i\mid i=1,\dots,n-1)$, each generated by picking one variable from each pair $\{x_i,y_i\}$.  Therefore $I$ has (pure) height $n-1$.  Height and depth of an ideal are equal in a Cohen-Macaulay ring, so the generators for $I$ must form a regular sequence, as desired.

We now show $R$ is a domain, by showing it is isomorphic to a \emph{semigroup ring} (see Chapter 7 of \cite{miller+sturmfels}).  We first claim the variables $x_i, y_i$, for $i=1,\dots,n-1$, are not zerodivisors on $R$. Since $R$ is Cohen-Macaulay it suffices to show each is a homogeneous parameter.  Fix $i$ and suppose we kill a minimal prime, $P$, of $(x_i)$.  The minor $zu_{ii}-x_iy_i$ is in $(x_i)\subset P$, so $P$ must also contain either $u_{ii}$ or $z$.  If $z\in P$ then we already know the dimension of $R$ drops upon killing $P$, since $z$ is a parameter.  On the other hand, suppose $u_{ii}\in P$.  Then the relations $u_{ii}u_{jj}-u_{ij}u_{ji}=0$ imply, for each $j\neq i$, either $u_{ij}$ or $u_{ji}$ is in $P$.  By the induction hypothesis none of these variables are zerodivisors, so again, the dimension of $R$  drops upon killing $P$ and we are done.

Having shown $x_i,y_i$ are not zerodivisors on $R$, we next observe $R$ injects into its localization at any subset of the variables $\{x_i, y_i \mid i=1,\dots,n-1\}$.  Fixing $i$, if we invert either $x_i$ or $y_i$ then we can use the principal 2-minor relations to solve for the other.  The same arguments held for the smaller matrix $X'$, so we may invert, say, all of the entries below the diagonal of $X$, then solve for the ones above the diagonal.  The resulting $K$-algebra, isomorphic to $R$, is generated by the $\binom{n+1}{2}$ indeterminates on or below the diagonal of $X$, along with monomials of the form $u_{ii}u_{jj}u_{ji}^{-1}$ for $i<j$, and monomials $u_{ii}zx_i^{-1},u_{ii}zy_i^{-1}$ for each $i$. Therefore $R$ is a semigroup ring. 
\end{proof}

\begin{thm}\label{thm:2normal} 
For all $n$, $K[X]/\pr_2$ is normal. 
\end{thm}

\begin{proof} 
Let $J$ denote the defining ideal of the singular locus for $R=K[X]/\pr_2$.  Serre's condition says if $J$ has depth at least 2, then $R$ is normal.  Since $R$ is Cohen-Macaulay (Theorem \ref{thm:p2X}) it is enough to show $J$ has height at least 2.  Let $J'$ denote the ideal generated by the degree $n$ monomials whose factors consist of exactly one variable from each pair $\{x_{ij},x_{ji}\}, 1\leq i\neq j\leq n$.  Then $J'\subseteq J$, since for all such monomials $\mu$, $R[\frac{1}{\mu}$] is a localized polynomial ring.  We will show each of the minimal primes of $J'$ contains some height 2 ideal $(x_{ij},x_{ji})$. 

Let $P$ be a minimal prime of $J'$.  If for each pair $\{x_{ij},x_{ji}\}$ we can choose one element not in $P$, multiply these choices together to get an element $u\in J'\setminus P$, a contradiction.  Therefore there exists some pair $\{x_{ij},x_{ji}\}\subset P$, $i\neq j$.  Now suppose we kill that pair.  The principal 2-minor relation ($x_{ii}x_{jj}-x_{ij}x_{ji}$) implies either $x_{ii}$ or $x_{jj}$ must also vanish; without loss of generality, say $x_{ii}$.  Then, for each $k\neq i,j$, the relation $x_{ii}x_{kk}-x_{ik}x_{ki}=0$ implies either $x_{ik}$ or $x_{ki}$ must vanish.  

We now count the drop in the dimension from $\dim K[X]=n^2$, upon killing $\pr_2$ and $(x_{ij},x_{ji})\subset P$.  Killing $x_{ij},x_{ji},x_{ii}$ as above, along with the $n-2$ variables $x_{ik}$ or $x_{ki}$ for each $k\neq i,j$ drops the dimension by at least $3+(n-2)=n+1$.  The binomials $x_{ii}x_{kk}-x_{ik}x_{ki}$ now vanish automatically.  The remaining binomials defining $\pr_2$ are exactly those not involving variables with $i$ in the index.  There are $\binom{n-1}{2}$ such, and since $\pr_2$ is a complete intersection the dimension goes down to 
\[
n^2-(n+1)-\binom{n-1}{2}=\binom{n+1}{2}-2=\dim R-2,
\] 
as desired. 
\end{proof}

\begin{cor}\label{cor:f-reg} 
For all $n$, $K[X]/\pr_2$ is strongly $F$-regular, and hence, $F$-regular.  
\end{cor}

\begin{proof} 
Since a normal ring generated by monomials is a direct summand of a regular ring, this follows from \cite{hochster/72} and \cite{hochster+huneke/90}. 
\end{proof}

% % % % %
\section{Using Matrix Rank to Find Minimal Primes of Principal Minor Ideals}

Describing $\pr_2$ relied on the fact that its generators, the principal 2-minors, are binomials.  Of course, once $t>2$, the generators for $\pr_t$ are not binomial and another strategy is required.  As it turns out, for any $t$, the components of $\Var(\pr_t)$ are stratified according to the rank, $r$, of a generic element.  For fixed $n,r,t$ with $1\leq t,r\leq n$, let $\mathscr Y_{n,r,t}\subset\Var(\pr_t)$ denote the locally closed subset of matrices with rank exactly $r$.  We study the components of $\mathscr Y_{n,r,t}$ and take their Zariski closures, $\overline{\mathscr Y}_{n,r,t}$, in $\Var(\pr_t)$.  The components of $\Var(\pr_t)$ will be among those closures as we vary $r$, since they are irreducible closed sets whose union is $\Var(\pr_t)$.  The issue is which ones are maximal.  In this section we shall analyze the sets $\mathscr Y_{n,n,t}$ and then apply the results to $\Var(\pr_{n-1})$.

% % %
\subsection{Rank $r=n$}\label{sec:rEqn}

We have a convenient way to describe components of $\mathscr Y_{n,n,t}$, which follows from a classical theorem stated and proved in Sir Thomas Muir's 1882 text, \emph{A Treatise on the Theory of Determinants}.  We first introduce some notation.  For any $n\times n$ matrix $A$, suppose $\underbar i,\underbar j\subset\{1,\dots,n\}$ are indexing sets of cardinality $t$:
\[
\begin{split}
\underbar i&=\{i_1,\dots,i_t\} \\
	\underbar j&=\{j_1,\dots,j_t\}
\end{split}
\]
We use $A(\underbar i;\underbar j)$ to denote the submatrix in $A$ indexed by the rows $\underbar i$ and columns $\underbar j$, and 
\[
A_{\underbar i,\underbar j}=(-1)^{\sigma}\det\left[A(\{1,\dots,n\}\setminus\underbar i;\{1,\dots,n\}\setminus\underbar j)\right],
\]
where $\sigma=i_1+\cdots+i_t+j_1+\cdots+j_t$, to denote the $(\underbar i,\underbar j)$th cofactor of $A$.  The \emph{cofactor matrix} of $A$, denoted $\cof(A)$, is the matrix whose $(i,j)$th entry is the cofactor $A_{\{i\},\{j\}}$.  We shall often abuse notation and write $A_{ij}=A_{\{i\},\{j\}}$.  Recall, the \emph{classical adjoint} of $A$, denoted $\Adj(A)$, is the transpose of the cofactor matrix. 

\begin{thm*}[\cite{muir}, \S 96] 
Let $A$ be an $n\times n$ matrix.  Suppose $\mu$ is a size $t$ minor of $\Adj A$, indexed by the rows $\underbar i$ and columns $\underbar j$ of $\Adj A$.  Then
\begin{equation}\label{eq:muir}
\mu=(\det A)^{t-1}\cdot A_{\underbar i,\underbar j}.
\end{equation}
\end{thm*}

Let $S=K[X]$ and $\Delta=\det X$, and write $S_{\Delta}=S[\frac{1}{\Delta}]$.  

\begin{thm}\label{thm:invBij} 
In the localized ring $S_{\Delta}$, the $K$-algebra automorphism $X\mapsto X^{-1}$ induces an isomorphism $\mathscr Y_{n,n,t}\cong\mathscr Y_{n,n,n-t}$.  
\end{thm}

\begin{proof} 
The set $\mathscr Y_{n,n,t}$ is the subscheme of $\Spec(S_{\Delta})\cong\GL(n,K)$, the general linear group of order $n$ over $K$, defined by the vanishing of the ideal $\pr_tS_{\Delta}$.  The general linear group has an automorphism, $g\mapsto g^{-1}$ for each $g\in\GL(n,K)$, which induces a $K$-algebra automorphism, $\Phi:S_{\Delta}\to S_{\Delta}$, which sends the entries of $X$ to the respectively indexed entries of $X^{-1}$:
\[
\Phi:x_{ij}\mapsto(-1)^{i+j}\frac{1}{\Delta}X_{ji}
\]
It is clear $\Phi$ is its own inverse.  By Muir's theorem (Equation \eqref{eq:muir}), each $t\times t$ minor of $X$ is mapped to the complementarily indexed $(n-t)\times(n-t)$ minor of $X^{-1}$, multiplied by $\Delta^{1-t}$.  Hence $\pr_tS_{\Delta}\leftrightarrow\pr_{n-t}S_{\Delta}$.  And,
\[
\frac{S_{\Delta}}{\pr_tS_{\Delta}}\cong\frac{S_{\Delta}}{\pr_{n-t}S_{\Delta}}
\]
induces the isomorphism on the respective schemes. 
\end{proof}

\begin{cor}\label{cor:Qn-1} 
The locally closed set $\mathscr Y_{n,n,n-1}$ is irreducible, with codimension $n$. 
\end{cor}

\begin{proof} 
By Theorem \ref{thm:invBij} it is enough to look at 
\[
\mathscr Y_{n,n,1}\cong\Spec \frac{S_{\Delta}}{\pr_1S_{\Delta}},
\]
which is clearly irreducible of codimension $n$. 
\end{proof}

\begin{cor}\label{cor:Qn-2} 
The locally closed set $\mathscr Y_{n,n,n-2}$ is irreducible, with codimension $\binom{n}{2}$. 
\end{cor}

\begin{proof} 
Apply Theorem \ref{thm:p2X} to $\mathscr Y_{n,n,2}$: Since $\Var(\pr_2)$ is irreducible of codimension $\binom{n}{2}$, so is its open subset, $\mathscr Y_{n,n,2}\subset\Var(\pr_2)$. 
\end{proof}

% % % 
\subsection{Principal $(n-1)$-Minors Case}\label{sec:n-1}

In this section we shall assert $n\geq 4$, which suffices in studying the minimal primes for $\pr_{n-1}$, since $\pr_2(X_{3\times 3})$ and $\pr_1(X_{2\times 2})$ are both prime.  It turns out that when $n\geq 4$, the determinantal ideal, $\I_{n-1}$ (generated by all size $(n-1)$ minors of $X$), is a minimal prime for $\pr_{n-1}$ (this is part of the statement of Theorem \ref{thm:n-1}).  To see $\I_{n-1}$ cannot be the only minimal prime, note the following examples.

\begin{ex}\label{ex:otherMinPrime} 
Say $n=4$.  The matrix 
\[
\left(
\begin{smallmatrix}
0 & 0 & 0 & 1 \\
	1 & 1 & 1 & 0 \\
	0 & 1 & 1 & 0 \\
	0 & 0 & 1 & 0 
	\end{smallmatrix}
	\right)\in\mathscr Y_{4,4,3}\subset\Var(\pr_3)
	\]
is of full rank and its principal 3-minors vanish.  Therefore, $\Var(\I_3)\subsetneq\Var(\pr_3)$ as algebraic sets.  
\end{ex}

\begin{ex}\label{ex:otherMinPrime2} 
In general, given a permutation $\tau$ of $\{1,\dots,n\}$ with no fixed points, its corresponding matrix has zeros on the main diagonal.  Hence, the matrix of $\tau^{-1}$ has full rank and its $(n-1)$-size principal minors vanish.  Since $\tau$ and $\tau^{-1}$ have the same fixed points, this applies to $\tau$ as well. 
\end{ex}

We shall see the only other minimal prime for $\pr_{n-1}$ is the defining ideal for the closure of $\mathscr Y_{n,n,n-1}$.  In other words, we claim the contraction of the ideal $\ker\Phi=\pr_{n-1}S_{\Delta}$ to $S$, as in Theorem \ref{thm:invBij}, is a minimal prime for $\pr_{n-1}$.  Let $\mathfrak Q_{n-1}=\mathfrak Q_{n-1}(X)$ denote this contraction.  Let $S=K[X]$, $\Delta=\det X$, and $S_{\Delta}=S[\frac{1}{\Delta}]$, as in Theorem \ref{thm:invBij}.  For any set $\mathscr Y\in\Spec S$, let $\overline{\mathscr Y}$ denote its closure.  

\begin{thm}\label{thm:n-1} 
For all $n\geq 4$, the minimal primes of $\pr_{n-1}$ are exactly $\I_{n-1}$ and $\mathfrak Q_{n-1}$. 
\end{thm}

\begin{proof} 
An immediate observation we make is when $r<t$, the matrices in $\mathscr Y_{n,r,t}$ have rank strictly less than $t$, so are in $\Var(\I_r)\subset\Var(\I_t)$, where $\I_r,\I_t$ are determinantal ideals.  In other words, if $r<t$, then any associated prime for the defining ideal of the closure of $\mathscr Y_{n,r,t}$ must also contain $\I_t$.  

The proof proceeds as follows: First, we show $\mathfrak Q_{n-1}$ is indeed a minimal prime for $\pr_{n-1}$.  Then, we show $\mathfrak Q_{n-1}$ and $\I_{n-1}$ are incomparable.  It then remains to analyze the case where a point $A\in\Var(\pr_{n-1})$ has rank $r=n-1$.  We will show the components containing $A$ in that case are embedded components in $\Var(\pr_{n-1})$.  From that, we then conclude $\I_{n-1}$ is the only other minimal prime for $\pr_{n-1}$.  

Let $D_{\Delta}$ denote the distinguished open set in $\Spec S$ consisting of invertible matrices.  We have 
\[
\left(\Var(\pr_{n-1})\cap D_{\Delta}\right)\subseteq\Var(\mathfrak Q_{n-1})\subseteq\Var(\pr_{n-1}).
\]
Furthermore, by Theorem \ref{thm:invBij}, 
\[
\Var(\mathfrak Q_{n-1})=\overline{\Var(\pr_{n-1})\cap D_{\Delta}}\cong\overline{\Var(\pr_1)\cap D_{\Delta}}
\] 
is exactly the closure in $\Spec S$ of the set of invertible matrices whose inverses have all zeros on the diagonal.  On the other hand, $\Var(\pr_{n-1})\cap D_{\Delta}$ is dense in $\Var(\pr_{n-1})$, so
\[
\dim\left(\Var(\pr_{n-1})\cap D_{\Delta}\right)=\dim(\pr_{n-1}).
\]
Any prime strictly contained in $\mathfrak Q_{n-1}$ must have height smaller than $\hgt \mathfrak Q_{n-1}$, and so cannot contain $\pr_{n-1}$.  Therefore $\mathfrak Q_{n-1}$ is a minimal prime for $\pr_{n-1}$.

We now show $\I_{n-1}$ and $\mathfrak Q_{n-1}$ are incomparable.  When $n>4$ we clearly cannot have $\pr_{n-1}\subseteq\I_{n-1}\subseteq\mathfrak Q_{n-1}$, because $\mathfrak Q_{n-1}$ is a minimal prime, and $\hgt\I_{n-1}=4\neq n$.  The difference in height also shows why we cannot have $\mathfrak Q_{n-1}\subseteq\I_{n-1}$ for $n>4$.  When $n=4$, $\I_{n-1}$ and $\mathfrak Q_{n-1}$ each have height 4 so since they are prime, containment between them occurs if and only if they are equal.  However, Example \ref{ex:otherMinPrime2} exhibits a matrix in $\Var(\mathfrak Q_{n-1})\setminus\Var(\I_{n-1})$, so $\I_{n-1}\neq\mathfrak Q_{n-1}$.  

We look for any additional minimal primes of $\pr_{n-1}$, according to rank.  Choose $A\in\Var(\pr_{n-1})$.  If $\rank A$ is $n$, then $A\in\Var(\mathfrak Q_{n-1})$.  If $\rank A<n-1$, then the $(n-1)$-minors of $A$ must vanish, so $A\in\Var(\I_{n-1})$.  It remains to find the components of $\Var(\pr_{n-1})$ containing $A$ when $\rank A=n-1$.  We claim any such component is not defined by a minimal prime.  This will also imply $\I_{n-1}$ is minimal, since $\I_{n-1}$ and $\mathfrak Q_{n-1}$ are incomparable.

Say $\Var(\pr_{n-1})$ has a component, $\mathscr Y\neq\Var(\mathfrak Q_{n-1})$, and suppose $A\in\mathscr Y$ has rank $n-1$, so $\mathscr Y\neq\I_{n-1}$.  Then there exist $\underbar i\neq\underbar j$ such that $\det A(\underbar i;\underbar j)\neq 0$, an open condition.  Thus there exists a non-empty open set $\mathscr U\subseteq\mathscr Y$, still irreducible, on which $\det B(\underbar i;\underbar j)\neq 0$ for all $B\in\mathscr U$, and $\dim(\mathscr U)=\dim(\mathscr Y)$.  We note how, if $\tau$ is a size $n$ permutation matrix, $\tau\transpose$ its transpose, then the action $X\mapsto\tau\cdot X\cdot\tau\transpose$ performs the same permutation on the rows of $X$ as it does the columns, and hence preserves $\pr_t$.  So we assert, without loss of generality, that $\underbar i=\{1,\dots,n-1\}$ and $\underbar j=\{2,\dots,n\}$.  Now factor $B\in\mathscr U$: 
\begin{equation}\label{eq:n-1Factor}
B=\begin{tikzpicture}[baseline]
	\footnotesize
	\matrix[matrix of math nodes,
	left delimiter=(,right delimiter=)] (C) {
		0 & c_2 & \cdots & c_{n-1} \\
		0 & c_2 & \cdots & c_{n-1} \\
		0 & c_2 & \cdots & c_{n-1} \\
		0 & c_2 & \cdots & c_{n-1} \\ [.2pc]
		0 & c_2 & \cdots & c_{n-1} \\
		};
	\begin{pgfonlayer}{main}
		\node[fill=black!5,rounded corners,
		inner sep=0pt,
		fit=(C-1-1) (C-4-1) (C-4-4) (C-1-4)] {$\mathbf I_{(n-1)\times(n-1)}$}; 
	\end{pgfonlayer}
	\begin{pgfonlayer}{background}
		\node[fill=black!5,rounded corners,
		inner sep=0pt,
		fit=(C-5-1) (C-5-4)] {};
	\end{pgfonlayer}
	\end{tikzpicture}
B(\underbar i;\underbar j)	
\begin{tikzpicture}[baseline]
	\footnotesize
	\matrix[matrix of math nodes, 
	left delimiter=(,right delimiter=)] (C) {
		c_1' & [.2pc] c_1' & c_1' & c_1' & c_1'\\
		\vdots & & & \\
		c_{n-2}' & & & \\
		0 & 0 & 0 & 0 & 0 \\
		}; 
	\begin{pgfonlayer}{main}
		\node[fill=black!5,rounded corners,
		inner sep=0pt,
		fit=(C-1-2) (C-1-5) (C-4-2) (C-4-5)] {$\mathbf I_{(n-1)\times(n-1)}$};
	\end{pgfonlayer}
	\begin{pgfonlayer}{background}
		\node[fill=black!5,rounded corners,
		inner sep=0pt,
		fit=(C-1-1) (C-3-1) (C-4-1)] {};
	\end{pgfonlayer}
	\end{tikzpicture}.
\end{equation}

The remaining $n-2$ principal minor conditions force $n-2$ of the parameters $c_2,\dots,c_{n-1}$, $c_1',\dots, c_{n-2}'$ to vanish.  What is left are $n-2$ non-zero parameters, along with the $(n-1)^2$ parameters that are the entries of $B(\underbar i;\underbar j)$.  We have 
\[
\dim \mathscr U\leq(n-1)^2+(n-2)=n^2-(n+1).
\]
But $\pr_{n-1}$ has $\binom{n}{n-1}=n<n+1$ generators, so the closure $\overline{\mathscr U}=\mathscr Y$ cannot be a component.  
\end{proof}  

\begin{cor}\label{cor:n-1} 
For $n\neq 3$, $\hgt\pr_t\leq\binom{n+1}{2}-\binom{t+2}{2}+4$.
\end{cor}

\begin{proof} 
We estimate the height by killing variables in $K[X]$, as in Figure \ref{fig:newBound}.  
\begin{figure}
\centering
\includegraphics[width=0.8\linewidth]{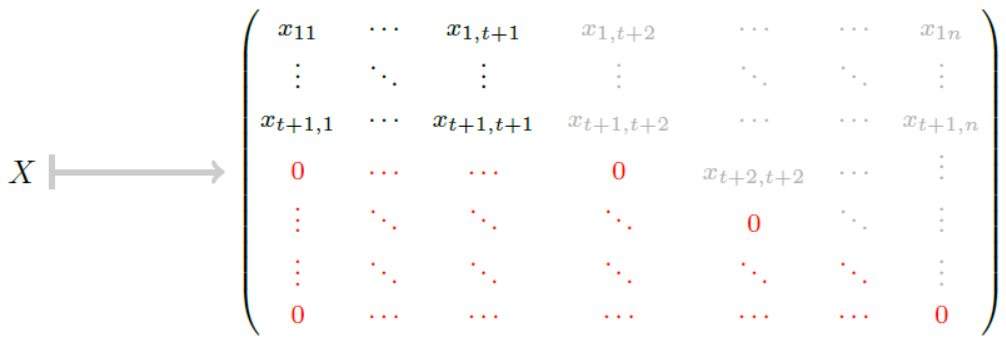}
\caption{Estimation of $\hgt\left(\pr_t(X)\right)$ from Corollary \ref{cor:n-1}.  It is enough to kill the first $t'$ entries of the $t'$th row, $t+2\leq t'\leq n$, to ensure the principal minors involving that row vanish.}
\label{fig:newBound}
\end{figure}
If we first kill the last row of $X$ then any principal minor involving that row, and hence, the last column, must vanish.  Therefore, if we want the principal minors involving the second-to-last row to vanish, it is enough to kill the first $n-1$ entries.  We may continue this argument inductively until we get to the $(t+1)$st row, having killed $n+(n-1)+\cdots+n-(n-t-2)=\binom{n+1}{2}-\binom{t+2}{2}$ variables so far.  When $n\geq 4$, Theorem \ref{thm:n-1} says the $t$-minors of the upper left $(t+1)\times(t+1)$ submatrix of $X$ have height 4, and that is independent of the variables we already killed, so we get the desired bound in that case.  In the cases where $n=1,2$ we may directly compute the height to see it satisfies the desired bound.  \end{proof}

% % %
\section{Special Case: $\pr_3(X_{4\times 4})$}\label{sec:ci}

In general, Theorem \ref{thm:n-1} gives us $\pr_{n-1}\subseteq\I_{n-1}\cap \mathfrak Q_{n-1}$.  However, computations in \emph{Macaulay2} show, in several prime characteristics, that equality holds for $n=4$.  We shall show, in fact, that $\pr_3(X_{4\times 4})$ is reduced in \emph{all} characteristics.  As a consequence, $\I_3(X_{4\times 4})$ and $\mathfrak Q_3(X_{4\times 4})$ are linked and we list relevant corollaries of this fact.  Throughout this section, unless specified otherwise, $n=4$.  We first need the following more general lemma.

\begin{lem}\label{lem:Qminors} 
For any $n$, $\mathfrak Q_{n-1}$ does not contain any size $r<n-1$ minors of $X$. 
\end{lem}

\begin{proof} 
Choose an $r\times r$ minor of $X$, indexed by the rows $\underbar i=\{i_1,\dots,i_r\}$ and columns $\underbar j=\{j_1,\dots,j_r\}$ of $X$.  We shall exhibit an $n\times n$ matrix $A\in\Var(\mathfrak Q_{n-1})$, whose $\left(\underbar i,\underbar j\right)$-minor is non-zero.  If we choose $A$ as a permutation matrix of the identity matrix $\vect I_{n\times n}$, then since permutation matrices are orthogonal, it shall suffice to construct a matrix $A'=A\transpose$, the transpose of $A$, whose main diagonal is all zeros, and whose $\left(\underbar j,\underbar i\right)$-minor does not vanish.  

We construct the submatrix $A'(\underbar j;\underbar i)$ by setting it equal to a permutation matrix of $\vect I_{r\times r}$ such that any entries on the main diagonal of $A'$ are zero.  Then in $A'$ put zeros in the remaining entries in the columns $\underbar i$.  Now complete the standard basis of column vectors, permuting the remaining columns so that the entries on the main diagonal of $A'$ are zero.  
\end{proof}

\begin{thm}\label{thm:red} 
Suppose $n=4$.  Then $\pr_{n-1}=\pr_{3}$ is reduced, and hence $\pr_{3}=\I_{3}\cap \mathfrak Q_3$. 
\end{thm}

\begin{proof} 
Theorem \ref{thm:n-1} implies $\pr_3$ is unmixed, so it suffices to show its primary decomposition is exactly $\I_3\cap \mathfrak Q_3$.  Also, we have the property $\pr_3$ is reduced if and only if its image at any localization is also reduced.  Let $\mu_{\bar i}\in\pr_3$ denote the principal minor obtained by omitting the $i$th row and column of $X$.  By inverting elements of $S=K[X]$, we first solve for variables using the equations $\mu_{\bar i}=0$.  Then we check the image of $\pr_3$ is reduced.  

Invert the minor $\delta_1=x_{11}x_{22}-x_{12}x_{21}$, which is not in $\I_3$ because $\I_3$ is generated by degree three polynomials, and which is not in $\mathfrak Q_3$ by Lemma \ref{lem:Qminors}.  Put $S_{\delta_1}=S[\frac{1}{\delta_1}]$, where $S=K[X]$.  We use $\mu_{\bar 4}=0$ to solve for $x_{33}$; let $F=\left.\mu_{\bar 4}\right|_{x_{33}=0}$, i.e., the determinant $\mu_{\bar 4}$, evaluated at $x_{33}=0$.  Then $F\equiv-x_{33}\delta_1\mod{\mu_{\bar 4}}$, and 
\[
\frac{S_{\delta_1}}{\pr_3S_{\delta_1}}\cong\frac{K[\frac{1}{\delta_1},\left.X\right|_{x_{33}=-\frac{F}{\delta_1}}]}{\left(\mu_{\bar 1},\mu_{\bar 2},\mu_{\bar 3} \mid x_{33}=-\frac{F}{\delta_1}\right)}.
\]
Similarly, put $G=\left.\mu_{\bar 3}\right|_{x_{44}=0}$.  Since $G$ is an expression independent of $F$, we have
\[
\frac{S_{\delta_1}}{\pr_3S_{\delta_1}}\cong\frac{K[\frac{1}{\delta_1},\left.X\right|_{x_{33}=-\frac{F}{\delta_1},x_{44}=-\frac{G}{\delta_1}}]}
{\left(\underbrace{ 
\begin{vmatrix}
x_{22} & x_{23} & x_{24} \\
	x_{32} & -\frac{F}{\delta_1} & x_{34} \\
	x_{42} & x_{43} & -\frac{G}{\delta_1}
	\end{vmatrix}}_{\left.\mu_{\bar 1}\right|_{x_{33}=-\frac{F}{\delta_1},x_{44}=-\frac{G}{\delta_1}}} , 
\underbrace{
\begin{vmatrix}
x_{11} & x_{13} & x_{14} \\
	x_{31} & -\frac{F}{\delta_1} & x_{34} \\
	x_{41} & x_{43} & -\frac{G}{\delta_1}\end{vmatrix}}_{\left.\mu_{\bar 2}\right|_{x_{33}=-\frac{F}{\delta_1},x_{44}=-\frac{G}{\delta_1}}} 
	\right)}.
	\]

We now solve for a variable not appearing in either polynomial
\[
\begin{split}F\equiv& x_{31}(x_{12}x_{23}-x_{13}x_{22})-x_{32}(x_{11}x_{23}-x_{13}x_{21})\mod{\mu_{\bar 4}} \\
G\equiv& x_{41}(x_{12}x_{24}-x_{14}x_{22})-x_{42}(x_{11}x_{24}-x_{14}x_{21})\mod{\mu_{\bar 3}},\end{split}\] 
using $\mu_{\bar 2}$.  Invert $\delta_2=x_{11}x_{34}-x_{14}x_{31}$, which, again, is not in $\I_3$ nor $\mathfrak Q_3$.  Then define $H=\left.\mu_{\bar 2}\right|_{x_{43}=0}$.  The image of $\pr_3$ is now principal:
\[\pr_3S_{\delta_1,\delta_2}\cong\left(
\begin{vmatrix}
x_{22} & x_{23} & x_{24} \\
	x_{32} & -\frac{F}{\delta_1} & x_{34} \\
	x_{42} & \frac{H}{\delta_2} & -\frac{G}{\delta_1}
	\end{vmatrix}
	\right)\cdot K\left[\frac{1}{\delta_1},\frac{1}{\delta_2},\left.X\right|_{x_{33}=-\frac{F}{\delta_1},x_{44}=-\frac{G}{\delta_1},x_{43}=\frac{H}{\delta_2}}\right]
	\]
Let $\gamma$ denote the generator of the image of $\pr_3S_{\delta_1,\delta_2}$, upon clearing denominators.  

The localized polynomial ring $S_{\delta_1,\delta_2}$ is a unique factorization domain, so it suffices to prove the irreducible factors of $\gamma\in S$ with non-zero image in $\I_3(S/\pr_3)$ or $\mathfrak Q_3(S/\pr_3)$ are square-free.  We factor $\gamma$ in the polynomial ring $S$, using \emph{Macaulay2}, where we put $K=\Z$ (which will imply the factorization is valid in all equicharacteristics):
\[
\gamma=\det X(\underbar i;\underbar j)f',
\]
where $\underbar i=\{1,2,3\}, \underbar j=\{1,2,4\}$, 
\[
f'=x_{31}(x_{12}x_{23}-x_{13}x_{22})Y_3+x_{14}(x_{21}x_{42}-x_{22}x_{41})\mu_{\bar 4}+\delta_1f,
\]
and $f$ is a degree 4 irreducible polynomial (see Equation \eqref{eq:f}). Modulo the ideal $\pr_3$, we get 
\[
\gamma\equiv\det X(\underbar i;\underbar j)\delta_1f\mod{\pr_3},
\]
which is square-free, as desired. 
\end{proof}

The polynomial $f$ in the proof of Theorem \ref{thm:red} is 
\begin{multline}\label{eq:f}
f=-x_{14}x_{21}x_{33}x_{42}+x_{11}x_{23}x_{34}x_{42}+x_{14}x_{22}x_{31}x_{43} \\
-x_{11}x_{22}x_{34}x_{43}-x_{12}x_{23}x_{31}x_{44}+x_{12}x_{21}x_{33}x_{44}.
\end{multline}

The fact that $\pr_3$ is a reduced complete intersection implies its minimal primes, $\I_3$ and $\mathfrak Q_3$, are linked; recall, two ideals $I,J$ in a Cohen-Macaulay ring $R$ are \emph{linked} (or \emph{algebraically linked}) means there exists a regular sequence $\vect f=f_1,\dots,f_h$ in $I\cap J$ such that $J=(\vect f):_RI$ and $I=(\vect f):_RJ$.  

\begin{rmk} 
Theorem \ref{thm:red} follows from the fact that $\I_3$ and $\mathfrak Q_3$ are \emph{geometrically linked}. 
\end{rmk}

For the remainder of this paper, we use statements from Proposition 2.5 and Remark 2.7 in \cite{huneke+ulrich/87} to pull some corollaries from Theorem \ref{thm:red}.  In our context, all results for local rings also hold for \emph{graded local} rings (see Chapter 1.5 of \cite{bruns+herzog} for justification for that statement). 

\begin{prop*}[\cite{peskine+szpiro}, 2.5 of \cite{huneke+ulrich/87}]  
Let $I$ be an unmixed ideal of height $h$ in a (not necessarily local) Gorenstein ring $R$, and let $\vect f=f_1,\dots,f_h$ be a regular sequence inside $I$ with $(\vect f)\neq I$, and set $J=(\vect f):I$.  
\begin{enumerate}[(a)]
\item $I=(\vect f):J$ (i.e., $I$ and $J$ are linked).
\item $R/I$ is Cohen-Macaulay if and only if $R/J$ is Cohen-Macaulay.
\item Let $R$ be local and let $R/I$ be Cohen-Macaulay.  Then $\omega_{R/J}\cong I/(\vect f)$ and $\omega_{R/I}\cong J/(\vect f)$, where $\omega_{R'}$ denotes the canonical modules for a Cohen-Macaulay ring $R'$.
\end{enumerate}
\end{prop*}

\begin{cor}\label{cor:QCM}
$\mathfrak Q_3$ is Cohen-Macaulay. 
\end{cor}

\begin{proof} 
For any $n$, determinantal ideals are known to be Cohen-Macaulay.  Linkage implies $\mathfrak Q_3$ must also be Cohen-Macaulay. 
\end{proof}

For the next corollary, we introduce an $\N^{2n}$-multigrading on $K[X]$: a polynomial has degree $(r_1,\dots,r_n;c_1,\dots,c_n)$ means under the standard grading, its degree in the variables from the $i$th row (resp., $j$th column) of $X$ is $r_i$ (resp. $c_j$), for all $i=1,\dots,n$ (resp., $j=1,\dots,n$).  Alternatively, $\deg x_{ij}=(\vect e_i;\vect e_j)$, the entries from the standard basis vectors for $K^n$.  We freely use that the sum, product, intersection, and colon of multigraded ideals is multigraded, as well as all associated primes over a multigraded ideal.  Observe how in the standard grading, the degree of any polynomial in the standard grading equals the sums $r_1+\cdots+r_n=c_1+\cdots+c_n$.  

\begin{cor}\label{cor:fQ} 
$\mathfrak Q_3=\pr_3+(f)$, where $f$ is as in Equation \eqref{eq:f}. 
\end{cor}

\begin{proof} 
Since $\I_3$ is Gorenstein (\cite{svanes}), linkage implies the canonical module, 
\[
\omega_{K[X]/\I_3}\cong\mathfrak Q_3/\pr_3,
\]
 is cyclic.  The proof of Theorem \ref{thm:red} shows the image of $f$ in $\mathfrak Q_3/\pr_3$ is non-zero.  It remains to show $f$ actually generates $\omega_{K[X]/\I_3}$.

We saw $f$ has degree 4 in the standard grading, so we show no polynomial of degree strictly less than 4 can generate $\mathfrak Q_3/\pr_3$.  Assume $g\in\mathfrak Q_3$ is such a polynomial.  Because $\pr_3:_{K[X]}\I_3=\mathfrak Q_3$, $g$ must multiply every generator $X_{ij}\in\I_3$ into $\pr_3$.  Suppose $g$ has degree 0 in the $i$th row.  The product of $g$ with any 3-minor not involving the $i$th row must then be a multiple of $X_{ii}\in\pr_3$.  If $g$ also has degree 0 in the $i$th column then either $g\in\pr_3$, a contradiction of the choice of $g$, or there exists another column $j\neq i$ where $g$ has degree 0 as well.  Then the product $gX_{ij}$ must simultaneously divide $X_{ii}$ and $X_{jj}$.  But this cannot happen, because the product has degree 0 in both the $i$th row and the $j$th column. 
\end{proof}

\begin{rmk} 
Corollary \ref{cor:fQ} also follows from Remark 2.7 in \cite{huneke+ulrich/87}. 
\end{rmk}

Recall, for general $n$, how we defined $\mathfrak Q_{n-1}$ as the contraction of $\ker\Phi$ to $K[X]$, where $\Phi$ is the map from Corollary \ref{cor:Qn-1}.  By definition, 
\[
\mathfrak Q_{n-1}=\pr_{n-1}:_{K[X]}\Delta^{\infty},
\]
where $\Delta=\det X$.

\begin{cor}\label{cor:definingQ} 
$\mathfrak Q_3=\pr_3:_{K[X]}\Delta$. 
\end{cor}

\begin{proof} 
From Corollary \ref{cor:fQ}, the only generator for $\mathfrak Q_3$ outside of $\pr_3$ is $f$, and direct computation shows $f\Delta\in\pr_3$. 
\end{proof}

\begin{conj}\label{conj:reduced} 
For all $n$, $\pr_{n-1}$ is reduced. 
\end{conj}

% % % % %
\section{Further Work}
Several difficulties arise in proving Conjecture \ref{conj:reduced}, the most immediate is that for $n>4$, the ideals $\pr_{n-1}$ are no longer unmixed.  Another difficulty is that for larger $n$ \emph{Macaulay2} in general cannot compute generators for $\frak Q_{n-1}$.  Current work by the author is in trying to subvert these obstacles.  

In general, imposing a rank condition on matrices in $\Var(\pr_t)$ gives a factorization like the one in Equation \eqref{eq:n-1Factor}.  When the rank $r$ is equal to $t$, the vanishing of principal minors becomes a statement about algebraic subsets of the products of Grassmannians $\Grass(t,n)$.  Future work will also focus on this approach.

% % % % % % % % % % % % % % % % % % % %
%%% Bibliography
% % % % % % % % % % % % % % % % % % % %

\bibliographystyle{amsplain}
\bibliography{../../wheelerbib}

\end{document}